\documentclass[a4paper, 12pt, final]{amsart} 
\usepackage{amscd}
\usepackage{mathrsfs}
\usepackage{}
\usepackage[all]{xy}
\usepackage[OT1]{fontenc}
\usepackage[latin1]{inputenc}
\usepackage{amssymb,latexsym}
\usepackage{type1cm,ae}
\usepackage[notref,notcite]{showkeys}
\usepackage{graphicx}
\usepackage{xspace}
\usepackage{setspace}
\usepackage[english]{babel}

\theoremstyle{plain}		\newtheorem{theorem}{Theorem}[section]
				\newtheorem{proposition}[theorem]{Proposition}
				\newtheorem{lemma}[theorem]{Lemma}
				
				\newtheorem{corollary}[theorem]{Corollary}

                \newtheorem{problem}[theorem]{Problem}

\numberwithin{equation}{section}

\newcommand*{\bR}{\ensuremath{\mathbb{R}}}

\newcommand*{\bC}{\ensuremath{\mathbb{C}}}

\newcommand*{\loc}{\mathrm{loc}}
\newcommand*{\closure}[1]{\overline{#1}}
\newcommand*{\bdary}[1]{\partial #1}

\newcommand*{\Wert}{\mathord{\mbox{|\kern-1.5pt|\kern-1.5pt|}}}
\newcommand*{\ie}{\mbox{i.e.}\xspace}

\DeclareMathOperator{\dist}{dist}
\DeclareMathOperator{\diam}{diam}

\DeclareMathOperator{\modulus}{mod}

\def\XXint#1#2#3{{\setbox0=\hbox{$#1{#2#3}{\int}$}
  \vcenter{\hbox{$#2#3$}}\kern-.5\wd0}}

\title[Generalized quasidisks and conformality II]{Generalized quasidisks and conformality II}

\author{Changyu Guo}
\address[Changyu Guo]{Department of Mathematics and Statistics, University of Jyv\"askyl\"a, P.O. Box 35, FI-40014 University of Jyv\"askyl\"a, Finland}
\email{changyu.c.guo@jyu.fi}

\thanks{C.Y.Guo was partially supported by the Academy of Finland grant 131477.}

\begin{document}
\begin{abstract}
We introduce a weaker variant of the concept of three point property, which is equivalent to a non-linear local connectivity condition introduced in~\cite{gkt13}, sufficient to guarantee
the extendability of a conformal map $f:\mathbb D\to\Omega$  to the entire plane as a
homeomorphism of locally exponentially integrable distortion. Sufficient conditions for extendability to a homeomorphism of locally $p$-integrable distortion are also given.
\end{abstract}

\maketitle

\section{Introduction}\label{sec:first}

One calls a Jordan domain $\Omega\subset\bR^2$ a quasidisk if it is
the image of the unit disk $\mathbb D$ under a quasiconformal mapping
$f\colon\bR^2\to\bR^2$ of the entire plane.
If $f$ is $K$-quasiconformal, we say that $\Omega$ is a $K$-quasidisk.
Another possibility is to require that $f$ is additionally conformal in
the unit disk $\mathbb D$. It is essentially due to K\"uhnau~\cite{k88} that
$\Omega$ is a $K$-quasidisk if and only if $\Omega$ is the image of $\mathbb D$
under a $K^2$-quasiconformal mapping $f:\bR^2\to \bR^2$ that is conformal
in $\mathbb D$, see~\cite{gh01}. The concept of a quasidisk is central in the theory of planar quasiconformal
mappings; see, for example, \cite{a06,aim09,g82,lv73}.

A substantial part of the theory of quasiconformal mappings has recently
been shown to extend in a natural form to the setting of mappings of
locally exponentially integrable
distortion~\cite{agrs10,aim09,d88,hek,hk03,iko02,oz05,z08}.
See Section 2 below for the definition of this class of mappings.
Thus one could say that $\Omega\subset\bR^2$ is a generalized quasidisk if it is
the image of the unit disk $\mathbb D$ under a homeomorphism $f\colon\bR^2\to\bR^2$ of the entire plane with locally exponentially integrable distortion. However, requiring that $f$ is additionally conformal in the unit disk $\mathbb D$ leads to different classes of domains, see~\cite[Theorem 1.1]{gkt13}. In this paper, a Jordan domain $\Omega\subset\bR^2$ is termed a generalized quasidisk if the additional conformality requirement is satisfied.

For a quasidisk $\Omega\subset\bR^2$, there are several equivalent characterizations. One of the simplest is Ahlfors~\cite{a63} three point property. Recall that a Jordan domain $\Omega\subset\bR^2$ has the three point property if there exists a constant $C\geq 1$ such that for each pair of distinct points $P_1,P_2\in\bdary\Omega$,
\begin{equation}\label{eq:three point property}
\min_{i=1,2}\diam(\gamma_i)\leq C|P_1-P_2|,
\end{equation}
where $\gamma_1,\gamma_2$ are components of $\bdary\Omega\backslash\{P_1,P_2\}$. In order to understand the geometry of generalized quasidisks, one naturally has to weaken the three point property. A Jordan domain $\Omega\subset\bR^2$ is said to have the three point property with a control function $\psi$ if there exists a constant $C\geq 1$ and an increasing function $\psi:[0,\infty)\to [0,\infty)$ such that for each pair of distinct points $P_1,P_2\in\bdary\Omega$, 
\begin{equation}\label{eq:three point property with control function}
\min_{i=1,2}\diam(\gamma_i)\leq \psi\Big(C|P_1-P_2|\Big).
\end{equation}
A closely related concept is the following $\psi$-local connectivity, which was introduced in~\cite{gkt13}. A domain  $\Omega\subset\bR^2$ is said to be $\psi$-locally connected if for each $x$ and all $r>0,$
\begin{itemize}
\item each pair of points in $B(x,r)\cap\Omega$ can be joined by an arc in
$B(x,\psi^{-1}(r))\cap\Omega$, and
\item each pair of points in $\Omega\backslash B(x,r)$ can be joined by an arc in
$\Omega\setminus B(x,\psi(r))$.
\end{itemize}
If we were to choose $\psi(t)=Ct,$ then this would reduce to the usual linear local
connectivity condition.
In Lemma~\ref{lemma:equivalence of three point property and LLC} below, we show that
that a Jordan domain $\Omega\subset\bR^2$ has the three point property with a control function $\psi$ if and only if $\Omega$ is $\psi^{-1}$-locally connected.

In~\cite[Theorem 1.2]{gkt13}, it was proved that if a Jordan domain $\Omega\subset\bR^2$ is $\psi$-locally connected with $\psi(t)=\frac{Ct}{\log^s\log\frac{1}{t}}$ for some positive constant $C$ and some $s\in (0,\frac{1}{4})$, then $\Omega$ is a generalized quasidisk. However, the result is not sharp regarding well-studied examples, see~\cite{gkt13}. In fact, the previous studies in~\cite{gkt13,kt07,kt10,t07} suggest that the critical case should be $\psi(t)=\frac{Ct}{\log\frac{1}{t}}$. Our first main result is the following generalization of~\cite[Theorem 1.2]{gkt13}.

\begin{theorem}\label{thm:thma}
If a Jordan domain $\Omega\subset\bR^2$ has the three point property with the control function $\psi(t)=Ct\log^{\frac{1}{2}}\frac{1}{t}$ for some positive constant $C$, then $\Omega$ is a generalized quasidisk.
\end{theorem}

Equivalently, Theorem~\ref{thm:thma} provides a general sufficient condition for extendability of a conformal mapping $f:\mathbb{D}\to \Omega$ to a homeomorphism of locally exponentially integrable distortion. It was pointed out in~\cite{gkt13} that this is essentially equivalent to extending the corresponding conformal welding to the whole plane as a homeomorphism of locally exponentially integrable distortion, see also Section 4 below.

Our second main result asserts that if we relax the control function $\psi$ to be a root in Theorem~\ref{thm:thma}, then we end up with a homeomorphism of the whole plane with locally $p$-integrable distortion.

\begin{theorem}\label{thm:thmb}
Let $\Omega\subset\bR^2$ be a Jordan domain that has the three point property with the control function $\psi(t)=t^{s}$, $0<s<1$.
Then any conformal mapping $f\colon \mathbb D\to\Omega$  can be extended to the entire plane as a homeomorphism of locally $p$-integrable distortion for all $p\in (0,\frac{s^2}{2(1-s^2)})$.
\end{theorem}

As pointed out in~\cite{gkt13}, (polynomial) interior cusps are more dangerous than (polynomial) exterior cusps in the locally exponentially integrable distortion case. Thus one expects that this is still the case for extensions with locally $p$-integrable distortion. Our next result confirms this expectation.
\begin{theorem}\label{thm:thmc}
Let $\Omega\subset\bR^2$ be a LLC-1 Jordan domain.
Then any conformal mapping $f\colon \mathbb D\to\Omega$  can be extended to the entire plane as a homeomorphism of locally $p$-integrable distortion for some $p>0$.
\end{theorem}

This paper is organized as follows. Section 2 contains the basic definitions and
Section 3 some auxiliary results. In Section 4, we study the relation of extending a Riemann mapping and the corresponding conformal welding. We prove our main results in Section 5.
In the final section, Section 6, we make some concluding remarks.

\section{Notation and Definitions}\label{sec:notdef}
We sometimes associate the plane $\bR^2$ with the complex plane $\bC$ for convenience and denote
by $\hat{\bC}$ the extended complex plane.
The closure of a set $U\subset\bR^2$ is denoted $\closure{U}$ and the boundary $\bdary{U}$.
The open disk of radius $r>0$ centered at $x\in\bR^2$ is denoted by $B(x,r)$ and we  simply write
$\mathbb D$ for the
unit disk.
The boundary of $B(x,r)$ will be denoted by $S(x,r)$ and the boundary of the unit disk $\mathbb D$
is written as $\partial \mathbb D.$
The symbol~$\Omega$ always refers to a domain, \ie a connected and open subset of~$\bR^2$.
We call a homeomorphism $f\colon\Omega\to f(\Omega)\subset\bR^2$ a homeomorphism of finite
distortion if $f\in W_{\loc}^{1,1}(\Omega;\bR^2)$ and
\begin{equation}\label{eq:disteq}
  \|Df(x)\|^2\leq K(x)J_f(x) \text{ a.e. in } \Omega,
\end{equation}
for some measurable function~$K(x)\geq1$ that is finite almost everywhere. Recall here that
$J_f\in L_{loc}^{1}(\Omega)$ for each homeomorphism $f\in W_{loc}^{1,1}(\Omega;\bR^2)$(cf.~\cite{aim09}).
In the distortion inequality~\eqref{eq:disteq}, $Df(x)$ is the formal differential of~$f$ at the
point~$x$ and $J_f(x):=\det Df(x)$ is the Jacobian. The norm of~$Df(x)$ is defined as
\begin{equation*}
  \|Df(x)\|:=\max_{e\in\bdary{\mathbb D}} |Df(x)e|.
\end{equation*}
For a homeomorphism of finite distortion it is convenient to write $K_f$ for the optimal distortion
function. This is obtained by setting $K_f(x)=\|Df(x)\|^2/J_f(x)$ when $Df(x)$ exists and $J_f(x)>0$,
and $K_f(x)=1$ otherwise. The distortion of~$f$ is said to be locally $\lambda$-exponentially
integrable if $\exp(\lambda K_f(x))\in L_{\loc}^1(\Omega)$, for some $\lambda>0$.
Note that if we assume $K_f(x)$ to be bounded, $K_f\le K,$
we recover the class of $K$-quasiconformal mappings,
see~\cite{aim09} for the theory of quasiconformal mappings.

Recall that a domain $\Omega$ is said to be linearly locally connected (LLC)
if there is a constant $C\geq 1$ so that
\begin{itemize}
\item (LLC-1) each pair of points in $B(x,r)\cap \Omega$ can be joined by an arc in
$B(x,Cr)\cap \Omega$, and
\item (LLC-2) each pair of points in $\Omega\backslash B(x,r)$ can be joined by an arc
in $\Omega\setminus B(x,C^{-1}r)$.
\end{itemize}
We need a weaker version of this condition, defined as follows.
We say that $\Omega$ is $(\varphi,\psi)$-locally connected ($(\varphi,\psi)$-LC) if
\begin{itemize}
\item ($\varphi$-LC-1) each pair of points in $B(x,r)\cap\Omega$ can be joined by an arc in
$B(x,\varphi(r))\cap\Omega$, and
\item ($\psi$-LC-2) each pair of points in $\Omega\backslash B(x,r)$ can be joined by an arc in
$\Omega\setminus B(x,\psi(r))$,
\end{itemize}
where $\varphi, \psi:[0, \infty)\to[0, \infty)$ are smooth increasing functions such that
$\varphi(0)=\psi(0)=0$, $\varphi(r)\geq r$ and $\psi(r)\leq r$ for all $r>0$.
For technical reasons, we assume that the function $t\mapsto \frac t {\varphi^{-1}(t)^2}$ is decreasing
and that there exist constants $C_1,C_2$ so that $C_1\varphi(t)\le \varphi(2t)\le C_2\varphi(t)$
and $C_1\psi(t)\le \psi(2t)\le C_2\psi(t)$ for all $t>0.$
If
$\varphi^{-1}=\psi$ above, as in the introduction, $\Omega$ will simply be called $\psi$-LC.
One could relax joinability by an arc above to joinability by a continuum, but this leads to
the same concept; see~\cite[Theorem 3-17]{hy88}.  Notice that if $\Omega$ is simply connected and
bounded, then $\varphi$-LC-1 guarantees that $\Omega$ is a Jordan domain.

Finally we define the central tool for us -- the modulus of a path family.
A Borel function $\rho\colon\bR^2\to[0,\infty\mathclose]$ is said to be
admissible for a path family $\Gamma$ if
$\int_\gamma\rho\,ds\geq1$ for each locally rectifiable $\gamma\in\Gamma$.
The modulus of the path family $\Gamma$ is then
$$
  \modulus(\Gamma):=
\inf\Big\{\int_{\Omega}\rho^2(x)\,dx :  \rho
\text{ is admissible for }  \Gamma\}.
$$
For subsets $E$ and $F$ of $\closure{\Omega}$ we write $\Gamma(E,F,\Omega)$ for
the path family consisting of all locally rectifiable paths joining~$E$ to~$F$
in~$\Omega$ and abbreviate $\modulus(\Gamma(E,F,\Omega))$ to $\modulus(E,F,
\Omega).$
In what follows, $\gamma(x,y)$ refers to a curve or an arc from $x$ to $y$.

When we write $f(x)\lesssim g(x)$, we mean that $f(x)\leq Cg(x)$ is 
satisfied for all $x$ with some fixed constant $C\geq 1$. Similarly, the 
expression $f(x)\gtrsim g(x)$ means that $f(x)\geq C^{-1}g(x)$ is satisfied 
for all $x$ with some fixed constant $C\geq 1$. We write $f(x)\approx g(x)$ 
whenever $f(x)\lesssim g(x)$ and $f(x)\gtrsim g(x)$.

\section{Auxiliary results}

We begin this section by showing the equivalence of the generalized three point property and generalized local connectivity mentioned in the introduction.
\begin{lemma}\label{lemma:equivalence of three point property and LLC}
Let $\Omega\subset\bR^2$ be a Jordan domain. Then $\Omega$ has the three point property with the control function $\psi$ if and only if $\Omega$ is $\psi^{-1}$-locally connected.
\end{lemma}
\begin{proof}
First, suppose that $\Omega$ has the three point property with the control function $\psi$. We want to show that $\Omega$ is $\psi$-LC-1. To this end, let $x,y\in B(z,r)\cap\Omega$. We may assume that there exist $x',y'\in B(z,r)\cap \bdary\Omega$ such that
\begin{equation*}
d(x,x')=d(x,\bdary\Omega), d(y,y')=d(y,\bdary\Omega)
\end{equation*}
and that $x$ can be connected to $x'$ by an arc $\beta_1$ in $\closure{\Omega}\cap B(z,r)$ and $y$ can be connected to $y'$ by an arc $\beta_2$ in $\closure{\Omega}\cap B(z,r)$.
Let $\alpha_1$ and $\alpha_2$ be the components of $\bdary\Omega\backslash \{x',y'\}$. We may assume that $\alpha_1\leq \alpha_2$. Then
\begin{equation*}
\diam(\alpha_1)\leq \psi(|x'-y'|)\leq \psi(2r).
\end{equation*}
Hence, $\gamma=\beta_1\cup\alpha_1\cup\beta_2$ is a curve that connects $x$ and $y$ in $\closure{\Omega}$ with diameter less than $2\psi(2r)$. Then the Jordan assumption for $\Omega$ implies that we may connect $x$ to $y$ in $\Omega$ by a curve with diameter no more than $3\psi(2r)$. This together with Lemma~\ref{lemma:consequence of technical assumption} below implies that $\Omega$ is $\psi$-LC-1. Similarly, one can prove that $\bR^2\backslash \closure{\Omega}$ is $\psi$-LC-1. Then the duality result in~\cite{gk12} implies that $\Omega$ is $\psi$-LC.

Next, we assume that $\Omega$ is $\psi^{-1}$-LC. Then, again by the duality result in~\cite{gk12}, we know that both $\Omega$ and $\bR^2\backslash\closure{\Omega}$ are $\psi$-LC-1. Let $x,y\in \bdary\Omega$ and let $\alpha_1$, $\alpha_2$ be the components of $\bdary\Omega\backslash \{x,y\}$. We may assume that $\diam(\alpha_1)\leq \diam(\alpha_2)$.
Let $z=\frac{x+y}{2}$ and $r=|x-y|$. Then $x,y\in B(z,r)$. We may choose two points $x'$ and $y'$ in $\Omega\cap B(z,r)$ such that $x$ can be connected to $x'$ by an arc $\beta_1$ in $\closure{\Omega}\cap B(z,r)$ and $y$ can be connected to $y'$ by an arc $\beta_2$ in $\closure{\Omega}\cap B(z,r)$. Then we may connect $x'$ to $y'$ by an arc $\gamma$ in $\Omega\cap B(z,2\psi(r))$. Then the curve $\eta=\beta_1\cup\gamma\cup\beta_2$ forms a crosscut of $\Omega$ with diameter no more than $4\psi(r)$. Similarly, we may form a crosscut $\eta'$ of $\bR^2\backslash\closure{\Omega}$ with diameter no more than $4\psi(r)$. Thus $\eta\cup\eta'$ is a Jordan curve which contains the Jordan arc $\alpha_1$. Therefore, the diameter of $\alpha_1$ is no more than $8\psi(r)$. This together with Lemma~\ref{lemma:consequence of technical assumption} below implies that $\Omega$ has the three point property with the control function $\psi$.
\end{proof}

\begin{lemma}[Lemma 3.5, \cite{gk12}]\label{lemma:consequence of technical assumption}
Let $C_1\geq 1$, $C_2\geq 1$, and $C_3\geq 1$ be given. There exists a constant $C$, depending only on $C_0$, $C_1$, $C_2$ and $C_3$, such that
\begin{equation}\label{eq:consequence of technical assumption}
C_1\varphi(C_2t)+C_3t\leq \varphi(Ct)
\end{equation}
for all $t>0$. Above, $C_0$ is the doubling constant of $\varphi^{-1}$.
\end{lemma}

The following two modulus estimates are standard,
see e.g. \cite{v88}.
\begin{lemma}\label{lemma:modululower}
Let $E, F$ be disjoint nondegenerate continua in $B(x,R).$ 
Then
\begin{equation}
C_0^{-1}\frac{1}{\log(1+t)}\geq \modulus(E,F,B(x,R))\geq C_0\frac{1}{\log(1+t)},
\end{equation}
where $t=\frac{\dist(E,F)}{\min\{\diam E, \diam F\}}$ and $C_0$ is an absolute constant.
\end{lemma}

\begin{lemma}\label{lemma:modulusupper}
Let $\Gamma$ be a curve family such that for all $t\in (r,R)$, the circle $|z-z_1|=t$ contains a curve $\gamma\in \Gamma$. Then
\begin{equation}
\modulus(\Gamma)\geq \frac{1}{2\pi}\log\frac{R}{r}.
\end{equation}
\end{lemma}

Next, we recall the following result on the modulus of continuity of a
quasiconformal mapping.
The proof can be found in~\cite{kot01}; also see \cite{g13}.

\begin{lemma}\label{lemma:modulus of continuity}
Suppose $g\colon \Omega\to \mathbb D$ is a $K$-quasiconformal mapping from a simply connected domain
$\Omega$ onto
the unit disk. Then there exists a positive constant $C$, (depending on $f$),
such that for any $\omega,\xi\in \Omega$,
\begin{equation}
|g(\omega)-g(\xi)|\leq Cd_{I}(\omega,\xi)^{\frac{1}{2K}},
\end{equation}
where $d_{I}(\omega,\xi)$ is defined as $\inf_{\gamma(\omega,\xi)\subset \Omega}\diam(\gamma(\omega,\xi))$.
In particular, if $\Omega$ above is $\varphi$-LC-1, then
\begin{equation}
|g(\omega)-g(\xi)|\leq C\varphi(|\omega-\xi|)^{\frac{1}{2K}}.
\end{equation}
\end{lemma}

Finally, we need the following key estimate.
\begin{lemma}\label{lemma:key modulus estimate}
Let $\Omega\subset\bR^2$ be a Jordan domain that has the three point property with the control function $\psi$. Let $\alpha_1$ and $\alpha_2$ be two disjoint arcs in $\bdary\Omega$ and let $\Gamma$ and $\Gamma'$ be the family of curves which join $\alpha_1$ and $\alpha_2$ in $\Omega$ and $\bR^2\backslash\closure{\Omega}$, respectively. If $\modulus(\Gamma)\leq C$, then
\begin{equation}\label{eq:diam bounds}
\min\{\diam(\alpha_1),\diam(\alpha_2)\}\leq \psi\circ\psi(d(\alpha_1,\alpha_2))
\end{equation}
and hence
\begin{equation}\label{eq:modulus bounds}
\modulus(\Gamma')\leq C_0^{-1}\log^{-1}\Big(1+\frac{\psi^{-1}\circ\psi^{-1}(\min\{\diam(\alpha_1),\diam(\alpha_2)\})}{\min\{\diam(\alpha_1),\diam(\alpha_2)\}} \Big).
\end{equation}
\end{lemma}
\begin{proof}
The idea of the proof is similar to that of the proof of Theorem 5.1 in~\cite{gkt13}. Let $\alpha_1$ and $\alpha_2$ be two disjoint arcs in $\bdary\Omega$. Choose $z_1\in \alpha_1$, $z_2\in\alpha_2$ so that
\begin{equation*}
|z_1-z_2|=d(\alpha_1,\alpha_2):=d.
\end{equation*}
Without loss of generality, we may assume that
\begin{equation*}
r:=\diam(\alpha_1)\leq \diam(\alpha_2).
\end{equation*}

Our aim is to show that $r\leq 2\psi\circ\psi(d)$. Thus we may clearly assume that $r>2\psi\circ\psi(d)$. Note that our assumption on $\psi$ implies that $r>\psi(d)$. Since $\Omega$ has the three point property with the control function $\psi$,
\begin{equation*}
\min_{i=1,2}\diam(\gamma_i)\leq \psi(d),
\end{equation*}
where $\gamma_1$, $\gamma_2$ are the components of $\bdary\Omega\backslash \{z_1,z_2\}$. Again, we may assume that
\begin{equation*}
\diam(\gamma_1)\leq \psi(d).
\end{equation*}
Let $\beta_1$, $\beta_2$ be the components of $\bdary\Omega\backslash (\alpha_1\cup\alpha_2)$, labeled so that $\beta_i\subset\gamma_i$. Then $\beta_1\subset\gamma_1\subset\closure{B}(z_1,\psi(d))$. Choose $z_0\in\beta_2$ and let $\delta_1$, $\delta_2$ denote the components of $\bdary\Omega\backslash \{z_0,z_1\}$ labeled so that $\alpha_2\subset\delta_2$. Then the fact $\Omega$ has the three point property with the control function $\psi$ implies that
\begin{equation*}
\min_{i=1,2}\diam(\delta_i)\leq \psi(|z_1-z_0|).
\end{equation*}
Choose $\omega_1,\omega_2\in\alpha_1$ so that
\begin{equation*}
r=|\omega_1-\omega_2|=\diam(\alpha_1).
\end{equation*}
Then $\omega_i\in \gamma_1\cup\delta_1$, and the fact that $\diam(\gamma_1)\leq \psi(d)<r$ implies that not both of these points can lie in $\gamma_1$. If $\omega_1\in \gamma_1$, then
\begin{align*}
\diam(\delta_1)&\geq |\omega_2-z_1|\geq |\omega_1-\omega_2|-|z_1-\omega_1|\\
&\geq r-\diam(\gamma_1)\geq r-\psi(d)\geq \frac{r}{2}.
\end{align*}
If both lie in $\delta_1$, then
\begin{equation*}
\diam(\delta_1)\geq |\omega_1-\omega_2|=r.
\end{equation*}
Thus
\begin{equation*}
\frac{r}{2}\leq \min_{i=1,2}\diam(\delta_i)\leq \psi(|z_1-z_0|).
\end{equation*}
It follows that
\begin{equation*}
\beta_2\cap B(z_1,\psi^{-1}(\frac{r}{2}))=\emptyset.
\end{equation*}
In particular, the circle $|z-z_1|=t$ separates $\beta_1$ and $\beta_2$ for $t\in (\psi(d),\psi^{-1}(\frac{r}{2}))$ and hence must contain an arc $\gamma$ joining $\alpha_1$ and $\alpha_2$ in $\Omega$. Thus Lemma~\ref{lemma:modulusupper} implies that
\begin{equation*}
\frac{1}{2\pi}\log\frac{\psi^{-1}(r/2)}{\psi(d)}\leq \modulus(\Gamma)\leq C
\end{equation*}
from which the claim follows. The desired inequality~\eqref{eq:modulus bounds} follows from Lemma~\ref{lemma:modululower} directly.

\end{proof}

\section{Extension of a conformal welding}

Before stating the main result of this section, let us describe the standard way of extending a conformal map
$f\colon \mathbb D\to\Omega,$ where $\Omega$ is a Jordan domain, to a mapping
of the entire plane.
First of all,
$f$ can be extended to a homeomorphism between $\closure{\mathbb D}$ and
$\closure{\Omega}$.
For simplicity, we denote this extended homeomorphism also by $f$.
It follows from the Riemann Mapping Theorem that there exists a
conformal mapping $g\colon \bR^2\setminus{\closure{\mathbb D}}\to\bR^2\setminus{\closure{\Omega}}$
such that the complement of the closed unit disk gets mapped to the complement
of $\overline \Omega$.
In this correspondence the boundary curve
$\Gamma=\bdary{\Omega}$ is mapped homeomorphically onto the boundary circle
$\partial \mathbb D$ and hence the composed mapping $G=g^{-1}\circ f$ is a
well-defined circle
homeomorphism, called conformal welding.
Suppose we are able to extend $G$ to the exterior of the
unit disk, with the extension
still denoted by $G$. Then the mapping $G'=g\circ G$ will be well-defined
outside the unit disk and
it coincides with $f$ on the boundary circle $\partial \mathbb D$.
Finally, if we define
\begin{equation*}
  F(x)=
  \begin{cases}
    G'(x) & \text{if } |x| \geq 1 \\
    f(x) & \text{if } |x| \leq 1 ,
  \end{cases}
\end{equation*}
then we obtain an extension of $f$ to the entire plane.
In the case of a quasidisk, that is when  $\Omega$ is linearly locally
connected (LLC), the extension $G$ can be chosen to be quasiconformal
and hence the obtained map $F$ is also quasiconformal.

On the other hand, the extendability
of a conformal mapping $f:\mathbb D\to \Omega$ to a homeomorphism $\hat f:\bR^2
\to \bR^2$ of locally integrable distortion is essentially equivalent to being
able to extend the conformal welding $G'$ above to this class. Indeed,
if $\hat f$ extends $f$, then $g^{-1}\circ \hat f$ extends $G$ to the exterior
of $\mathbb D$ and has the same distortion as $\hat f.$ Reflecting (twice)
with respect to the unit circle one then further obtains an extension
to $\mathbb D\setminus \{0\}.$ Hence, one obtains an extension $\hat G'$ of
$G'$ to $\bR^2\setminus \{0\}$ with distortion that has the same local 
integrability degree as the distortion of $\hat f.$ If the latter distortion
is sufficiently nice in a neighborhood of infinity (e.g. bounded), then this
holds in all of $\bR^2$ as well.

Given a sense-preserving homeomorphism $f\colon \partial \mathbb D\to \partial \mathbb D$
and $0< t< \frac{\pi}{2}$, set
\begin{equation}
\delta_f(\theta,t)=\max \Big\{\frac{|f(e^{i(\theta+t)})-f(e^{i\theta})|}{|f(e^{i\theta})-f(e^{i(\theta-t)})|}, \frac{|f(e^{i(\theta-t)})-f(e^{i\theta})|}{|f(e^{i\theta})-f(e^{i(\theta+t)})|}\Big\}.
\end{equation}
Clearly $\delta_f$ is continuous in both variables, $\delta_f\geq 1$ and $\delta_f(\theta+2\pi,t)=\delta_f(\theta,t)$. The scalewise distortion of $f$ is defined as $\rho_f(t)=\sup_{\theta}\delta_f(\theta,t)$.

In the following, we discuss a standard way of extending a conformal welding $G:\bdary\mathbb{D}:\to\bdary\mathbb{D}$ to a global homeomorphism of the whole plane with controlled distortion. More precisely, we want to present the following result, which is implicitly contained in~\cite[Section 2 and Section 3]{z08}.

\begin{proposition}\label{prop:results for conformal welding}
Given a conformal welding $G:\bdary\mathbb{D}:\to\bdary\mathbb{D}$, there exists a homeomorphism $\hat{G}:\bR^2\to \bR^2$ with the following property:
\begin{itemize}
\item For some $\delta\in (0,\frac{1}{2})$, $\hat{G}(z)=z$ if $|z|<\delta$ or $|z|>\frac{1}{\delta}$.
\item The distortion of $\hat{G}$ is bounded above by the scalewise distortion of $G$, \ie
\begin{equation}\label{eq:estimate of scalewise distortion}
K_{\hat{G}}(z)\leq C\rho_{G}(\log |z|)=C\sup_{\theta\in [0,2\pi]}\delta_{G}(\theta,\log |z|), 
\end{equation}
if $\delta\leq |z|\leq \frac{1}{\delta}$ for some absolute constant $C>0$.
\end{itemize}
\end{proposition}

Let us describe below the argument leading to Proposition~\ref{prop:results for conformal welding}.
Given a conformal welding $G:\bdary\mathbb{D}\to\bdary\mathbb{D}$, we first want to extend $G$ to a homeomorphism $\tilde{G}:\closure{\mathbb{D}}\to\closure{\mathbb{D}}$. We may represent $G$ in the form
\begin{equation*}
G(e^{2\pi ix})=e^{2\pi ih(x)},
\end{equation*}
where $h:\bR\to \bR$ is a homeomorphism of the real line which commutes with the unit translation $x\mapsto x+1$. For simplicity, we may assume that $G(1)=1$ and hence $h(0)=0$.

Next, we extend our mapping $h$ to a homeomorphism $H:\mathbb{H}\to \mathbb{H}$. This can be done via the well-known Beurling-Ahlfors extension. To be more precise, for $0<y<1$, set
\begin{equation}\label{eq:B-A extension}
H(x+iy)=\frac{1}{2}\int_0^1 (h(x+ty)+h(x-ty))dt+i\int_0^1 (h(x+ty)-h(x-ty))dt.
\end{equation}
Then $H=h$ on the real axis and $H$ is a $C^1$-smooth homeomorphism of $\mathbb{H}$. Since $h(x+1)=h(x)+1$, for $y=1$
\begin{equation*}
H(x+i)=x+i+C_0,
\end{equation*}
where $C_0=\int_0^1 h(t)dt-\frac{1}{2}\in [-\frac{1}{2},\frac{1}{2}]$. For $1\leq y\leq 2$ we extend $H$ linearly by setting
\begin{equation*}
H(z)=z+(2-y)C_0, \quad z=x+iy.
\end{equation*}
Finally, we set $H(z)=z$ if $y=Im(z)\geq 2$. It is easy to check that $H(z+k)=H(z)+k$ for $k\in \mathbb{Z}$.
We set
\begin{equation}
\tilde{G}=\textbf{e}\circ H\circ L,
\end{equation}
where $\textbf{e}:z\mapsto e^{2\pi iz}$ is the lifting mapping and $L:z\mapsto \frac{\log z}{2\pi i}$ is the logarithmic mapping. We claim that $\tilde{G}:\mathbb{D}\backslash \{0\}\to \mathbb{D}\backslash \{0\}$ is a well-defined homeomorphism. To see thisu], let $z=re^{0i}=re^{2\pi i}$ be as in Figure~\ref{fig:homeomorphism}. We need to show that $L$ is well-defined on the segment $P:=\{z:r\leq |z|\leq 1\}$. Note that in Figure~\ref{fig:homeomorphism}, the vertical line $[0,L(z)]$ corresponds to the image of $P$ with argument $0$ and the vertical line $[1,L(z)+1]$ corresponds to the image of $P$ with argument $2\pi$ under the mapping $L$. Note also that $L(re^{0i})=\frac{\log r}{2\pi i}$ and $L(re^{2\pi i})=\frac{\log r+2\pi i}{2\pi i}$. Since $H$ satisfies that $H(z+1)=H(z)+1$ and $\textbf{e}$ is $1$-periodic, the mapping $\tilde{G}$ is a homeomorphism in the annulus $\mathbb{D}\backslash B(0,r)$. 
\begin{figure}[h]
  \includegraphics[width=12cm]{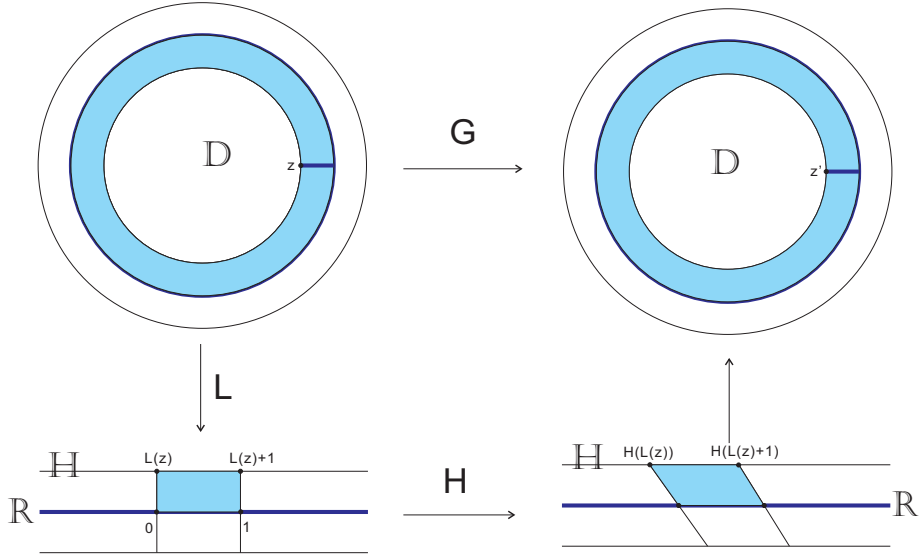}\\
  \caption{The homeomorphism $\tilde{G}$}\label{fig:homeomorphism}
\end{figure}
Moreover, $\tilde{G}=G$ on $\bdary\mathbb{D}$ and $\tilde{G}(z)=z$ for $0<|z|\leq \delta:= e^{-4\pi}$. Thus $\tilde{G}$ is well-defined homeomorphism of the unit disk if we additionally set $\tilde{G}(0)=0$. 

Finally, we may define our mapping $\hat{G}:\bR^2\to\bR^2$ by setting
\begin{equation*}
  \hat{G}(z)=
  \begin{cases}
    \tilde{G}(z) & \text{if } |z| \geq 1 \\
    R\circ\tilde{G}\circ R(z) & \text{if } |z| \leq 1 ,
  \end{cases}
\end{equation*}
where $R(z)=\frac{1}{\bar{z}}$ is the inversion with respect to the unit circle.  To complete the proof of Proposition~\ref{prop:results for conformal welding}, we need to estimate the distortion of $\hat{G}$. 

It is clear that we only need to estimate the distortion of $\tilde{G}$. Since $\textbf{e}$ and $L$ are conformal mappings, it follows that
\begin{align*}
K_{\tilde{G}}(z)=K_H(\omega),\quad z=e^{2\pi i\omega},\ \omega\in \mathbb{H}.
\end{align*}
So we reduce all distortion estimates for $\tilde{G}$ to the corresponding ones for $H$. Since $H$ is conformal for $y>2$ and linear for $y\in [1,2]$, it suffices to estimate $K_H$ in the strip $S=\bR\times [0,1]$. 
The desired estimate
\begin{equation*}
K_H(x+iy)\leq C_0\rho_h(y),\quad x+iy\in S
\end{equation*}
follows from the calculation in~\cite{cch96}, where
\begin{align*}
\rho_h(t)&=\sup_{\theta\in \bR}\delta_h(\theta,t)\\
&:=\sup_{\theta\in \bR}\max \Big\{\frac{|h(\theta+t)-h(\theta)|}{|h(\theta)-h(\theta-t)|},\frac{|h(\theta-t)-h(\theta)|}{|h(\theta)-h(\theta+t)|}\Big\}.
\end{align*}
Note that if $t\in [0,1]$, then
\begin{equation*}
\delta_{G}(\theta,t)\approx \delta_h(\theta,t)
\end{equation*}
and hence
\begin{equation*}
\rho_G(t)\approx\rho_h(t).
\end{equation*}
The proof of Proposition~\ref{prop:results for conformal welding} is complete.

As an application of Proposition~\ref{prop:results for conformal welding}, we easily obtain the following corollary. 

Let $\delta$ be as in Proposition~\ref{prop:results for conformal welding} and let  $\varepsilon<\delta$ is sufficiently small such that $$\log |z|\leq 2|r-1|$$ for $z=re^{i\theta}\in A:=\closure{B(0,1+\varepsilon)}\backslash B(0,1-\varepsilon)$. Proposition~\ref{prop:results for conformal welding} implies that 
\begin{align*}
  K_{\hat{G}}(z)\leq C\rho_G(\log |z|)\leq C\rho_G(2|r-1|)
\end{align*}
for $z\in A$. If $\rho_G(t)\leq C't^{-\alpha}$ as $t\to 0$, then
\begin{align*}
K_{\hat{G}}(z)\leq C |r-1|^{-\alpha},
\end{align*}
for $z\in A$. Integrating in polar coordinates, we immediately obtain the following corollary. 
\begin{corollary}\label{coro:sufficient for extending}
Let $G:\bdary\mathbb{D}:\to\bdary\mathbb{D}$ be a conformal welding. If 
\begin{equation*}
\rho_G(t)=O(\log\frac{1}{t}) \quad \text{as}\quad t\to 0,
\end{equation*}
then $G$ extends to a homeomorphism of the entire plane of locally exponentially integrable distortion. Moreover, if 
\begin{equation*}
\rho_G(t)=O(t^{-\alpha}) \quad \text{as}\quad t\to 0
\end{equation*}
for some $\alpha>0$,
then $G$ extends to a homeomorphism of the entire plane of locally $p$-integrable distortion with any $p\in (0,\frac{1}{\alpha})$.
\end{corollary}

\section{Main proofs}
Theorem~\ref{thm:thma} follows from the following more general result.
\begin{theorem}\label{thm: general thmc}
If $\Omega\subset\bR^2$ is a Jordan domain that has the three point property with a control function $\psi$ such that
\begin{equation}\label{eq:sufficient condition}
    \limsup_{r\to 0}\frac{r}{\psi^{-1}\circ\psi^{-1}(r)\log\frac{1}{r}}\leq C'\\
\end{equation}
for some constant $C'$, then $\Omega$ is a generalized quasidisk.
\end{theorem}

\begin{proof}[Proof of Theorem~\ref{thm: general thmc}]
The idea is similar to that used in~\cite[Theorem 5.1]{gkt13}. Since $\Omega$ is a Jordan domain, $f$ extends to a homeomorphism between $\mathbb D$
and $\closure{\Omega}$ and we denote also this extension by $f$.
Let $e^{i(\theta-t)},\ e^{i\theta}$ and $e^{i(\theta+t)}$ be three points on $S$.
Since $f$ is a sense-preserving homeomorphism, $f(e^{i(\theta-t)}),\ f(e^{i\theta})$ and
$f(e^{i(\theta+t)})$ will be on the boundary of $\Omega$ in order.
Let $g\colon \bR^2\setminus {\closure {\mathbb D}}\to\bR^2\setminus{\closure{\Omega}}$
be a conformal mapping from the Riemann Mapping Theorem. Then $g$ extends to a homemorphism
between $\bR^2\setminus{\mathbb D}$ and $\bR^2\setminus{\Omega}$.
As before, we still denote this extension by $g$. Based on the discussion in the previous section, we only need to estimate the scale-wise distortion of the conformal welding $G:=g^{-1}\circ f$. 

Let $P=e^{i(\theta+\pi)}$ be the anti-polar point of $e^{i\theta}$ on $\bdary\mathbb{D}$ and let $\gamma_f(P,\theta-t)$ denote the arc from $f(P)$ to $f(e^{i(\theta-t)})$ on $\bdary\Omega$ . There exists a $t_0$ small enough such that $\diam(\gamma_f(-1,\theta-t))\geq \diam(\gamma_f(\theta,\theta+t))$ when $t\in [0,t_0)$. Let $\Gamma_1$ be the family of curves in $\mathbb{D}$ joining $\gamma(P,e^{i(\theta-t)})$ and $\gamma(e^{i\theta},e^{i(\theta+t)})$. Then Lemma~\ref{lemma:modululower} implies that
\begin{equation}
\modulus(\Gamma_1)\leq C_1
\end{equation}
for some absolute constant $C_1>0$. The conformal invariance of modulus gives us that
\begin{equation}
\modulus(\Gamma):=\modulus(f(\Gamma_1))\leq C_2.
\end{equation}
Thus, we may use Lemma~\ref{lemma:key modulus estimate} for $\alpha_1=\gamma_f(\theta,\theta+t)$ and $\alpha_2=\gamma_f(P,\theta-t)$ and conclude that
\begin{equation*}
\diam(\gamma_f(\theta,\theta+t))\leq \psi\circ\psi(d),
\end{equation*}
where $d=d(\alpha_1,\alpha_2)$ is the distance between these two arcs. Moreover,
\begin{equation}\label{eq:modulus estimate for outer family}
\modulus(\Gamma')\leq C\log^{-1}\Big(1+\frac{\psi^{-1}\circ\psi^{-1}(\diam(\alpha_1))}{\diam(\alpha_1)} \Big),
\end{equation}
where $\Gamma'$ is the family of curves joining $\alpha_1$ and $\alpha_2$ in $\bR^2\backslash\closure{\Omega}$. Again by conformal invariance of modulus, we obtain that
\begin{equation}\label{eq:modulus from conformal invar}
\log^{-1}(1+\frac{1}{\delta_{G}(\theta,t)})\leq C_0^{-1}\modulus(\Gamma'),
\end{equation}
where $C_0$ is the constant from Lemma~\ref{lemma:modululower}. Note that 
\begin{equation*}
\frac{1}{\log(1+t)}\approx \frac{1}{t}\quad\text{as}\quad t\to 0.
\end{equation*}
Combining~\eqref{eq:modulus estimate for outer family} with~\eqref{eq:modulus from conformal invar} gives us the estimate
\begin{equation*}
\delta_{G}(\theta,t)\leq \frac{C\diam(\alpha_1)}{\psi^{-1}\circ\psi^{-1}(\diam(\alpha_1))} .
\end{equation*}
Since $\frac{t}{\psi^{-1}\circ\psi^{-1}(t)}$ is non-increasing, Lemma~\ref{lemma:modulus of continuity} implies that
\begin{equation}\label{eq:important estimate for scalewise distortion}
\delta_{G}(\theta,t)\leq \frac{Ct^2}{\psi^{-1}\circ\psi^{-1}(t^2)} .
\end{equation}
Theorem~\ref{thm: general thmc} follows immediately from~\eqref{eq:sufficient condition},~\eqref{eq:important estimate for scalewise distortion} and Corollary~\ref{coro:sufficient for extending}.

\end{proof}

\begin{proof}[Proof of Theorem~\ref{thm:thmb}]
This is basically contained in the proof of Theorem~\ref{thm: general thmc}. In this case, the desired bound~\eqref{eq:important estimate for scalewise distortion} reads as follows:
\begin{equation*}
\delta_{G}(\theta,t)\leq Ct^{2(1-\frac{1}{s^2})}.
\end{equation*}
The claim follows directly from Corollary~\ref{coro:sufficient for extending} with $\alpha=2(\frac{1}{s^2}-1)$.
\end{proof}

\begin{proof}[Proof of Theorem~\ref{thm:thmc}]
If $\Omega$ is LLC-1, then Lemma~\ref{lemma:modulus of continuity} implies that $f^{-1}$ is uniformly H\"older continuous. On the other hand, the duality result implies that $\bR^2\backslash\closure{\Omega}$ is LLC-2, which is further equivalent to $\bR^2\backslash\closure{\Omega}$ being John by the results in~\cite{nv91}. Then by the results in~\cite{kot01}, $g$ is also H\"older continuous. Hence $G^{-1}$ is uniformly H\"older continuous with some exponent $\alpha$. Therefore, for $t$ sufficiently small, we have
\begin{align*}
\delta_G(\theta,t)&\leq \max\Big\{\frac{|G(e^{i(\theta+t)})-G(e^{i\theta})|}{|G(e^{i\theta})-G(e^{i(\theta-t)})|}, \frac{|G(e^{i\theta})-G(e^{i(\theta-t)})|}{|G(e^{i(\theta+t)})-G(e^{i\theta})|}\Big\}\\
&\lesssim t^{-1/\alpha}.
\end{align*}
The claim follows from Corollary~\ref{coro:sufficient for extending}.
\end{proof}

\section{Concluding remarks}
\subsection{Definition of generalized quasidisks}

 This was previously discussed briefly in the introduction. Recall that $\Omega\subset\bR^2$ is a generalized quasidisk if it is
the image of the unit disk $\mathbb D$ under a homeomorphism $f\colon\bR^2\to\bR^2$ of the entire plane with locally exponentially integrable distortion and $f$ is conformal in the unit disk $\mathbb{D}$. However, this is not natural from the technical point of view since our extended mapping $\hat{f}$ is the identity outside a compact disk.

On the other hand, from the point view of conformal welding, requiring that $f$ is identity at infinity is reasonable since it makes the two extension problems equivalent as discussed in Section 4.

From the point view of finding a geometric characterization of generalized quasidisks, this additional requirement is also natural. More precisely, the geometry of a generalized quasidisk $\Omega\subset\bR^n$ should be determined by the geometry of its boundary (at least this is the case if $\Omega$ is a quasidisk). Intuitively the geometry of $\bdary\Omega$ should have nothing to do with the behavior of the global homeomorphism $f$ at infinity.

These observations suggest that it is better to require the global homeomorphism $f$ to be identity at $\infty$ in the definition of a generalized quasidisk.

\subsection{Inward pointing and outward pointing cusps}
As we already observed, (polynomial) interior cusps are more dangerous than (polynomial) exterior cusps for our extension problems. 
This is not a big surprise from the technical point of view since our aim is to estimate the scalewise distortion of our conformal welding $G$. It is fairly easy to observe that this is closely related to the modulus of continuity of $G^{-1}$. On the other hand, combining the duality results in~\cite{gk12} with the global H\"older continuity estimates of conformal mappings in~\cite{g13,kot01}, one can immediately see how the role of $\Omega$ being $\varphi$-LC-1 or $\psi$-LC-2 is related the modulus of continuity of $G$ and $G^{-1}$. In fact, this is exactly the way we proved~\cite[Theorem 4.4]{gkt13}.

\subsection{Open problems} 
To end the article, we put forward some open problems, which are plausible to be true.
\begin{problem}
In Theorem~\ref{thm:thma}, can we further relax the control function $\psi$ to be of the form $\psi(t)=Ct\log\frac{1}{t}$ ? By the result in~\cite{gkt13}, we know that the result fails for $\psi(t)=Ct\log^{1+\delta}\frac{1}{t}$ for any $\delta>0$.
\end{problem}

\begin{problem}
In Theorem~\ref{thm:thmc}, can we conclude that the extension has better integrability for the distortion, say locally exponentially integrable distortion, if we additionally assume that $\Omega$ is $\psi$-LC-2 for $\psi(t)=t^{s}$ with $s>1$?
\end{problem}

\begin{problem}
If we require reasonable good moduli of continuity for both $f$, $g$ and their inverses, say both $f$ and $g$ are bi-H\"older continuous up to boundary, can we conclude that $\Omega$ is a generalized quasidisk ? 
\end{problem}

\textbf{Acknowledgements}

I wish to thank my supervisor Academy Professor Pekka Koskela for many useful suggestions.

\end{document}